\documentclass[12pt, reqno]{amsart}
\usepackage[english]{babel}
\usepackage[letterpaper, top = 2cm, bottom = 2cm, left = 3cm, right = 3cm, marginparwidth = 1.75cm]{geometry}
\usepackage{amsmath, amssymb, amsfonts, amsthm}
\usepackage{graphicx}
\usepackage{tikz}
\usetikzlibrary{calc, arrows, decorations.text}
\usepackage{enumitem}
\usepackage{aliascnt}
\usepackage{centernot}

\usepackage[
colorlinks = true,
bookmarks = true,
unicode,
pagebackref = true,   
pdfauthor = {Ruofei Yao},
linkcolor = blue,     
citecolor = blue,
urlcolor = blue
]{hyperref}

\usepackage[nameinlink, capitalise, noabbrev]{cleveref}

\newtheorem*{rmks}{Remarks}

\newtheorem{thm}{Theorem}[section]

\newaliascnt{lma}{thm}
\newtheorem{lma}[lma]{Lemma}
\aliascntresetthe{lma}

\newaliascnt{prop}{thm}
\newtheorem{prop}[prop]{Proposition}
\aliascntresetthe{prop}

\newaliascnt{cor}{thm}

\aliascntresetthe{cor}

\newaliascnt{conj}{thm}
\newtheorem{conj}[conj]{Conjecture}
\aliascntresetthe{conj}

\newaliascnt{rmk}{thm}

\aliascntresetthe{rmk}

\newaliascnt{defn}{thm}

\aliascntresetthe{defn}

\newcommand{\R}{{\mathbb{R}}}

\numberwithin{equation}{section}
\tikzset{ > = stealth'}

\begin{document}

\title[Monotonicity properties of the Robin torsion function]{
Monotonicity properties of the Robin torsion function in a class of symmetric planar domains
}
\author{Qinfeng Li}
\address{School of Mathematics, Hunan University, Changsha, P. R. China. }
\email{liqinfeng1989@gmail.com}

\author{Juncheng Wei}
\address{Department of Mathematics, Chinese University of Hong Kong, Shatin, Hong Kong}
\email{wei@math.cuhk.edu.hk}

\author{Ruofei Yao}
\address{School of Mathematics, South China University of Technology, Guangzhou, P. R. China. }
\email{ruofeiyaopde@gmail.com; yaorf5812@126.com}

\begin{abstract}

We prove the monotonicity property of the Robin torsion function in a smooth planar domain $\Omega$ with a line of symmetry, provided that the Robin coefficient $\beta$ is greater than or equal to the negative of the boundary curvature $\kappa$, i.e., $\beta \geq - \kappa$ on $\partial\Omega$. We also construct a counterexample showing that the curvature condition cannot, in general, be removed.
\end{abstract}


\date{\today}
\maketitle


\section{Introduction}

In this paper, we study the following simple torsion equation 
\begin{equation} \label{eq101}
\begin{cases}
\Delta {u} = - 1 & \text{in } \Omega, \\
\partial_{\nu} {u} + \beta {u} = 0 & \text{on } \partial\Omega, 
\end{cases}
\end{equation}
where $\Omega \subset \mathbb{R}^{n}$ is a bounded Lipschitz domain, $\nu$ denotes the unit outward normal on $\partial\Omega$, and $\beta > 0$ is a constant. The solution to \eqref{eq101} is unique and is often referred to as the Robin torsion function. It can be viewed as the steady-state temperature under the condition of a uniformly distributed heat source inside $\Omega$, and the boundary condition models the convection heat transfer mode. 

From the variational point of view, the $\beta$-Robin torsion function is the minimizer of the following variational problem: 
\begin{equation} 
{J}_{\beta}(\Omega): \, = \inf \left\{
\frac{1}{2} \int_{\Omega} |\nabla {u}|^{2} d{x}
 + \frac{\beta}{2} \int_{\partial\Omega} {u}^{2} d\sigma
 - \int_{\Omega} {u} d{x}: \, {u} \in {H}^{1}(\Omega)
\right\}. 
\end{equation}
The critical shape for ${J}_{\beta}(\cdot)$ under a fixed volume constraint satisfies the following overdetermined system \cite{LYh23}: 
\begin{equation} \label{eq104overdet}
\begin{cases}
 - \Delta {u} = 1 & \text{in } \Omega, \\
\frac{\partial {u}}{\partial \nu} + \beta {u} = 0 & \text{on } \partial\Omega, \\
|\nabla {u}|^{2} - 2\beta^{2} {u}^{2} + \beta {u}^{2} {H} - 2{u} = \text{constant} & \text{on } \partial\Omega, 
\end{cases}
\end{equation}
where ${H}$ is the mean curvature of $\partial\Omega$. 
Recall that when $\beta = \infty$, \eqref{eq104overdet} becomes Serrin's seminal overdetermined system, and the solution must be a ball via the moving plane method, see \cite{Ser71}. On the other hand, for any $\beta > 0$ and under a volume constraint, it is known that the ball is the unique minimizer of ${J}_{\beta}(\cdot)$; see \cite{BG15} and \cite{ANT23} for two different proofs. Therefore, this motivates us to make the following conjecture: 

\begin{conj}\label{conj11}
For any $\beta>0$, if the overdetermined system \eqref{eq104overdet} admits a solution in $\Omega$, then $\Omega$ must be a ball.
\end{conj}

\autoref{conj11} provides a broad motivation for studying geometric properties of solutions to Robin boundary value problems. We emphasize, however, that the results of the present paper do not establish this conjecture. In particular, they neither derive the symmetry of the domain from the overdetermined condition \eqref{eq104overdet} nor provide a complete replacement for the moving-plane method.

In the Dirichlet case, the moving-plane method is a fundamental tool for proving symmetry and one-sided monotonicity of solutions. For Robin boundary conditions, however, this method is generally unavailable because the reflected functions do not satisfy compatible boundary conditions, and the required information on the direction of the boundary gradient is missing. We therefore address the following more specific structural question, which is motivated by, but logically independent of, \autoref{conj11}.

\smallskip
\noindent\textbf{Question A.}
Let $\Omega \subset \mathbb{R}^{2}$ be symmetric with respect to one coordinate axis and convex in the direction of the other coordinate axis. Must the Robin torsion function be monotone in each half of $\Omega$?

\smallskip
This question concerns a fundamental geometric property of solutions subject to Robin boundary conditions. As far as we are aware, no monotonicity result of this type was previously known for the Robin torsion function, even in symmetric convex domains.

Before stating our result on Question~A, we recall the related problem of level-set convexity. If $\Omega$ is convex, the level sets of the $\beta$-Robin torsion function are known to be convex for sufficiently large $\beta$; see \cite{CF21}. In contrast, level-set convexity may fail for small $\beta$; see \cite{ACH20}. Thus, Robin torsion
functions exhibit a level-set convexity-breaking phenomenon, although the precise threshold $\beta_*$ remains unknown. For the classical torsion function with Dirichlet boundary conditions, the level sets are always convex; see \cite{BL76, Ml71}. These results illustrate that the geometric properties of Robin and Dirichlet torsion functions may be similar for large $\beta$, but can differ substantially for small $\beta$.

The following theorem gives an affirmative answer to Question~A under an additional curvature condition.

\begin{thm}\label{thm12main}
Let $\Omega \subset \mathbb{R}^{2}$ be a bounded connected smooth domain that is symmetric with respect to the horizontal coordinate axis and convex in the vertical direction. Let ${u}$ be the $\beta$-Robin torsion
function in $\Omega$, namely, the solution of \eqref{eq101}, where $\beta > 0$. Assume that
\begin{equation}\label{eq106curvature}
\beta \geq - \min_{\partial\Omega} \kappa,
\end{equation}
where $\kappa$ denotes the signed curvature of $\partial\Omega$.
Then ${u}$ is symmetric with respect to the horizontal coordinate axis
and satisfies
\begin{equation}\label{eq107}
{x}_{2} \frac{\partial {u}}{\partial {x}_{2}} < 0 \quad \text{in } \overline{\Omega} \cap \{ {x}_{2} \neq 0 \}.
\end{equation}
\end{thm}

\autoref{thm12main} establishes one-sided monotonicity of the Robin torsion function for a class of symmetric, directionally convex planar domains. It provides a partial structural analogue of the monotonicity obtained by the moving-plane method in the Dirichlet setting. However, it neither derives domain symmetry from the overdetermined condition nor, by itself, implies the rigidity statement in \autoref{conj11}.

We make two further observations. First, for every convex planar domain symmetric with respect to an axis and every $\beta > 0$, the $\beta$-Robin torsion function is strictly decreasing away from the axis of symmetry. Second, the theorem also permits certain nonconvex domains, provided that $\min_{\partial\Omega}\kappa \geq - \beta$. To the best of our knowledge, this is the first monotonicity result for Robin torsion functions involving an explicit quantitative condition on the boundary curvature.

The proof of \autoref{thm12main} is based on the continuity method via domain deformation. 
The same approach also applies to a class of semilinear Robin problems, for which we establish an analogous
monotonicity result in \autoref{prop24b}. 
By adapting the proof, one can show that similar results hold for the Neumann torsion function, which is the solution (up to an additive constant) to the following equation: 
\begin{equation}
\begin{cases}
\Delta {u} = - 1 & \text{in } \Omega, \\
\partial_{\nu} {u} = - \frac{|\Omega|}{{P}(\Omega)} & \text{on } \partial\Omega, 
\end{cases}
\end{equation}

The condition \eqref{eq106curvature} arises naturally in \autoref{thm12main} from differentiating the Robin boundary condition, an essential step in our proof. One may ask whether this assumption can be removed, or whether it is optimal. The next theorem shows that symmetry alone does not guarantee monotonicity. 

\begin{thm} \label{thm13Example}
For any $\beta > 0$, there exists a smooth, nonconvex planar domain $\Omega$, symmetric with respect to the horizontal axis and convex in the vertical direction, such that the $\beta$-Robin torsion function ${u}$ in $\Omega$ \textbf{fails} to satisfy the monotonicity property \eqref{eq107}. 
\end{thm}

The construction relies on a careful symmetry analysis and local asymptotic expansions for a nonconvex polygon (see \autoref{prop31}); a small perturbation of this polygon yields a smooth counterexample. Hence, symmetry and convexity in one direction alone do not ensure the monotonicity property \eqref{eq107}, in sharp contrast with the Dirichlet case \cite{GNN79, BN91}. 

\textbf{Outline of the paper.} In \autoref{Sec2Mon}, we mainly prove \autoref{thm12main}, while the counterexample in \autoref{thm13Example} is constructed in \autoref{Sec3Example}. 


\section{The proof of the monotonicity property} \label{Sec2Mon}

In this section, we establish the monotonicity result for the $\beta$-Robin torsion function using the continuity method via domain deformation. 
We begin by recalling Serrin's lemma. 

\begin{lma} \label{lemS}
Let $\Omega$ be a domain in $\R^{n}$ with the origin ${O}$ lying on its boundary $\partial\Omega$. Suppose that, in a neighborhood of ${O}$, the boundary $\partial\Omega$ consists of two ${C}^{2}$ hypersurfaces $\{\rho = 0\}$ and $\{\sigma = 0\}$ which intersect transversally. Assume that $\rho < 0$ and $\sigma < 0$ in $\Omega$. Let ${w} \in {C}^{2}(\overline{\Omega})$ satisfy ${w} < 0$ in $\Omega$ and ${w}({O}) = 0$, and suppose that
\begin{equation*}
{L}{w} = {a}_{ij}{w}_{{x}_{i}{x}_{j}} + {b}_{i}{w}_{{x}_{i}} + {c}{w} \geq 0, 
\end{equation*}
where ${L}$ is a uniformly elliptic operator whose coefficients are uniformly bounded in absolute value. 
Assume that
\begin{equation} \label{eq201a}
{a}_{ij}\rho_{{x}_{i}}\sigma_{{x}_{j}} \geq 0 \mbox{ at } {O}. 
\end{equation}
If equality holds in \eqref{eq201a}, we further assume that ${a}_{ij} \in {C}^{2}$ in $\overline\Omega$ near ${O}$, and that
\begin{equation} \label{eq201b}
{D}({a}_{ij}\rho_{{x}_{i}}\sigma_{{x}_{j}}) = 0 \mbox{ at } {O}
\end{equation}
for any first-order derivative ${D}$ at ${O}$ tangent to the submanifold $\{\rho = 0\} \cap \{\sigma = 0\}$. 
Then, for any direction ${s}$ at ${O}$ which enters $\Omega$ transversally to each hypersurface, we have
\begin{enumerate} [label = $(\arabic*)$, start = 1]
\item
$\partial {w}/\partial {s} < 0$ at ${O}$ in the case of strict inequality in \eqref{eq201a}. 
\item
either $\partial{w}/\partial {s} < 0$ or $\partial^{2} {w}/\partial s^{2} < 0$ at $O$ in the case of equality in \eqref{eq201a}. 
\end{enumerate}
\end{lma}

This result is obtained in Gidas--Ni--Nirenberg \cite[Lemma S]{GNN79}, which can be viewed as an extension of Hopf's boundary lemma at a corner, due to Serrin \cite{Ser71}. 

We now introduce the notations used throughout this section: 
\begin{itemize}
\item
${s}$ denotes the arc length parameter along $\partial\Omega$, measured in the clockwise direction. 
\item
$\tau({x})$ denotes the unit tangent vector to $\partial\Omega$ at ${x} \in \partial\Omega$, oriented in the clockwise direction. 
\item
$\nu({x})$ denotes the unit \textbf{outer} normal vector of $\partial\Omega$ at ${x} \in \partial\Omega$. 
\item
$\kappa({x})$ stands for the standard curvature of $\partial\Omega$ at ${x} \in \partial\Omega$. 
\end{itemize}
Consequently, the frame $(\tau({x}), \nu({x}))$ forms a local right-handed rectangular coordinate system. The curvature $\kappa$ is nonnegative whenever the domain is smoothly convex. By the Frenet-Serret formulas, we have
\begin{equation*} 
\frac{\partial}{\partial {s}} \tau = - \kappa \nu \quad\text{and}\quad \frac{\partial}{\partial {s}} \nu = \kappa \tau. 
\end{equation*}
Differentiating the Robin boundary condition $\partial_{\nu} {u} + \beta {u} = 0$ on the boundary with respect to the arc length parameter ${s}$ (in the clockwise direction), we obtain
\begin{gather}
\label{eq201FrenetB}
{D}^{2}{u}[\tau, \nu] + (\kappa + \beta) \nabla {u} \cdot \tau = 0, 
\\ 
\label{eq201FrenetC}
{D}^{3}{u}[\tau, \tau, \nu] - \kappa {D}^{2}{u}[\nu, \nu] + (2\kappa + \beta) {D}^{2}{u}[\tau, \tau] - \kappa (\kappa + \beta) \nabla {u} \cdot \nu + \partial_{{s}}\kappa \nabla {u} \cdot \tau = 0. 
\end{gather}
These two identities will be utilized in subsequent arguments. 

To prove \autoref{thm12main}, the key point is, indeed, to establish the nonvanishing of the vertical derivative of ${u}$ on the boundary (excluding the vertices). 

\begin{prop} \label{prop22}
Let $\Omega \subset \R^{2}$ be a ${C}^{4}$ domain satisfying the following assumptions: 
\begin{enumerate}[label = \rm({A}\arabic*), start = 1]
\item \label{eq01itema}
$\Omega$ is a connected planar domain which is convex in the ${x}_{2}$ direction and symmetric with respect to the plane $\{ {x}_{2} = 0 \}$; 
\item \label{eq01itemb}
the curvature $\kappa$ of $\partial\Omega$ satisfies $\beta + \min_{\partial\Omega} \kappa \geq 0$. 
\end{enumerate}
Suppose that the $\beta$-Robin torsion function ${u}$ in $\Omega$ satisfies
\begin{equation}
\partial_{{x}_{2}} {u} < 0 \text{ in } \Omega \cap \{ {x}_{2} > 0 \}, 
\end{equation}
then ${u}$ also satisfies
\begin{gather} \label{eq203}
\partial_{{x}_{2}} {u} < 0 \text{ on } \partial\Omega \cap \{ {x}_{2} > 0 \}, 
\\ \label{eq204}
\partial_{{x}_{2}{x}_{2}} {u} < 0 \text{ on } \overline{\Omega} \cap \{ {x}_{2} = 0 \}. 
\end{gather}
\end{prop}

\begin{figure}[htp]\centering
\begin{tikzpicture}[scale = 1]
\pgfmathsetmacro\lha{4}; \pgfmathsetmacro\lhb{2}; 
\pgfmathsetmacro\lv{2}; 
\draw[] plot [domain = - 90 : 90] ({\lha*cos(\x)}, {\lv*sin(\x)}); 
\draw[] plot [domain = 90 : 270] ({\lhb*cos(\x)}, {\lv*sin(\x)}); 
\draw[ ->, dotted] ({ - \lhb*1.2}, 0) -- ({\lha*1.3}, 0) node [above] {\footnotesize ${x}_{1}$}; 
\draw[ ->, dotted] ({(\lha - \lhb)/2}, { - \lv*0.9}) -- ({(\lha - \lhb)/2}, {\lv*1.25}) node [right] {\footnotesize ${x}_{2}$}; 
\fill ( - \lhb, 0) circle (1pt) node [below left] {\footnotesize ${z}_{L}$}; 
\fill (\lha, 0) circle (1pt) node [below right] {\footnotesize ${z}_{R}$}; 
\pgfmathsetmacro\JiaoP{50}; 
\pgfmathsetmacro\xP{\lha*cos(\JiaoP)}; 
\pgfmathsetmacro\yP{\lv*sin(\JiaoP)}; 
\pgfmathsetmacro\SloP{(\yP/\lv^2)/(\xP/\lha^2)}; 
\pgfmathsetmacro\AngP{atan(\SloP)}; 
\draw[ -> ](\xP, \yP) -- ++ ({\AngP - 90} : {\lv*0.45}) node[right] {${y}_{1}$}; 
\draw[ -> ](\xP, \yP) -- ++ ({\AngP} : {\lv*0.45}) node[right] {${y}_{2}$}; 
\fill (\xP, \yP) circle (1pt) node [below left] {\scriptsize ${P}$}; 
\end{tikzpicture}
\caption{The domain $\Omega = \{{x}: |{x}_{2}| < \phi({x}_{1})\}$} \label{fig01}
\end{figure}
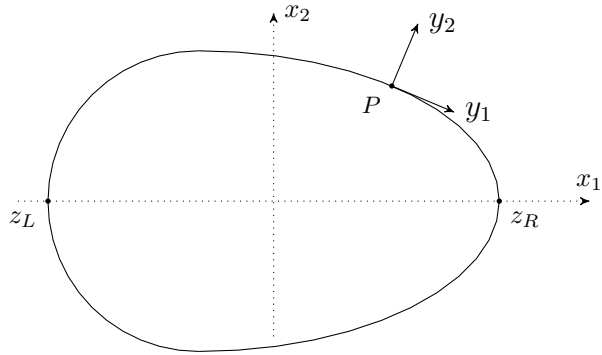

\begin{proof}
Since the solution ${u}$ is unique, it must be symmetric with respect to the line $\{ {x}_{2} = 0 \}$. The remainder of the proof is divided into two parts. 

\textbf{Part~1}. 
We establish the strict monotonicity on the boundary (excluding the vertices). 
Writing ${u}_{{x}_{2}}: = \partial {u}/\partial {{x}_{2}}$, we argue by contradiction
to prove that ${u}_{{x}_{2}} < 0$ on $\Gamma^{ + }:  = \partial\Omega \cap \{{x}_{2} > 0\}$. Suppose that
\begin{equation} \label{eq205}
{u}_{{x}_{2}} = 0 \quad \text{at } {P}
\end{equation}
for some point ${P} = (\bar{x}_{1}, \bar{x}_{2}) \in \Gamma^{ + }$. By the Robin boundary condition $\partial_{\nu} {u} = - \beta {u} < 0$ on $\partial\Omega$, we know that $\nu({P})$ cannot be vertical. Consider a Cartesian coordinate system $({y}_{1}, {y}_{2})$ with origin at ${P}$ so that the ${y}_{1}$-axis is the tangential direction (clockwise) and the positive ${y}_{2}$-axis is the outer normal direction $\nu({P})$ at ${P}$; see \autoref{fig02}. By assumption \ref{eq01itema} on the domain, there exists a unit vector $(\alpha_{1}, \alpha_{2})$ with $\alpha_{2} \geq 0$, $\alpha_{1} \neq 0$ such that
\begin{equation}
\partial_{{x}_{2}} = \alpha_{1}\partial_{{y}_{1}} + \alpha_{2}\partial_{{y}_{2}}. 
\end{equation}
Differentiating the Robin boundary condition along the boundary (see \eqref{eq201FrenetB}), one obtains
\begin{equation} \label{eq208a}
\partial_{{y}_{2}} {u} + \beta {u} = 0, \quad
\partial_{{y}_{1}{y}_{2}} {u} + (\kappa_{*} + \beta) \partial_{{y}_{1}} {u} = 0 \quad \text{at } {x} = {P}, 
\end{equation}
where $\kappa_{*} = \kappa({P})$ is the curvature of $\partial\Omega$ at ${P}$. Observe that $\partial_{{x}_{2}} {u}$ attains its local maximum zero at ${P}$. Differentiating $\partial_{{x}_{2}} {u} = \alpha_{1} \partial_{{y}_{1}} {u} + \alpha_{2} \partial_{{y}_{2}} {u}$ along the boundary $\partial\Omega$, we get
\begin{equation} \label{eq208b}
\alpha_{1} \partial_{{y}_{1}{y}_{1}} {u} + \alpha_{2} \partial_{{y}_{2}{y}_{1}} {u} = 0 \quad \text{at } {x} = {P}. 
\end{equation}
Applying the Hopf lemma to the harmonic function $\partial_{{x}_{2}} {u}$ at ${P}$, we obtain
\begin{equation} \label{eq208c}
\alpha_{1} \partial_{{y}_{1}{y}_{2}} {u} + \alpha_{2} \partial_{{y}_{2}{y}_{2}} {u} > 0 \quad \text{at } {x} = {P}. 
\end{equation}

First, consider the case $\alpha_{2} = 0$. Then, ${u}_{{x}_{2}}({P}) = 0$ implies ${u}_{{y}_{1}}({P}) = 0$. Combining this with \eqref{eq208a} gives $\partial_{{y}_{1}{y}_{2}} {u}({P}) = 0$. But then, since $\alpha_{2} = 0$, \eqref{eq208c} reduces to
\begin{equation*}
\alpha_{1} \partial_{{y}_{1}{y}_{2}} {u}({P}) > 0, 
\end{equation*}
which is a contradiction. 

Next, consider the case $\alpha_{2} > 0$. By \eqref{eq208a}, \eqref{eq208b}, and \eqref{eq208c}, 
\begin{equation}
\alpha_{1} \partial_{{y}_{1}{y}_{2}} {u} = - (\beta + \kappa_{*}) \alpha_{1} \partial_{{y}_{1}} {u} = (\beta + \kappa_{*}) \alpha_{2} \partial_{{y}_{2}} {u} = - \beta (\beta + \kappa_{*}) \alpha_{2} {u} \quad \text{at } {x} = {P}, 
\end{equation}
and
\begin{equation} \label{eq210}
\alpha_{1}^{2} \partial_{{y}_{1}{y}_{1}} {u}
 = - \alpha_{1}\alpha_{2} \partial_{{y}_{1}{y}_{2}} {u}
 = \beta (\beta + \kappa_{*}) \alpha_{2}^{2} {u} \geq 0 \; \; \text{at}\; \; {x} = {P}, 
\end{equation}
where assumption \ref{eq01itemb} is used. On the other hand, multiplying \eqref{eq208b} by $\alpha_{1}$ and subtracting it from $\alpha_{2}$ times \eqref{eq208c}, we obtain $\alpha_{1}^{2} \partial_{{y}_{1}{y}_{1}} {u}({P}) - \alpha_{2}^{2} \partial_{{y}_{2}{y}_{2}} {u}({P}) < 0$. Since $\alpha_{1}^{2} + \alpha_{2}^{2} = 1$, 
\begin{equation}
\partial_{{y}_{1}{y}_{1}} {u} < \alpha_{2}^{2} \Delta {u} < 0 \; \; \text{at}\; \; {y} = 0, 
\end{equation}
where the equation \eqref{eq101} for ${u}$ is used. This contradicts \eqref{eq210}. 

In both cases, we reach a contradiction. Hence, \eqref{eq203} is proved. 

\textbf{Part~2}. 
We now prove \eqref{eq204}. 
By the symmetry of ${u}$, 
\begin{equation} \label{eq211a}
\partial_{{x}_{2}} {u} = 0 \quad \text{on } \overline{\Omega} \cap \{ {x}_{2} = 0 \}. 
\end{equation}
Applying the Hopf lemma to the harmonic function $\partial_{{x}_{2}} {u}$, we obtain $\partial_{{x}_{2}{x}_{2}} {u} < 0$ on $\Omega \cap \{ {x}_{2} = 0 \}$. Consequently, $\partial_{{x}_{2}{x}_{2}} {u} \leq 0$ on $\partial\Omega \cap \{ {x}_{2} = 0 \}$. 

To show the negativity of the second tangential derivative of ${u}$ at both the left and right vertices, we argue by contradiction. Without loss of generality, suppose that
\begin{equation} \label{eq211b}
\partial_{{x}_{2}{x}_{2}} {u}({z}_{R}) = 0
\end{equation}
where ${z}_{L}$ and ${z}_{R}$ are the left and right vertices of $\partial\Omega$. Differentiating the Robin boundary condition twice along the boundary (see \eqref{eq201FrenetC}), and using \eqref{eq211a} and \eqref{eq211b}, we obtain
\begin{equation*}
\partial_{{x}_{1}{x}_{2}{x}_{2}} {u} - \kappa \partial_{{x}_{1}{x}_{1}} {u} - (\kappa + \beta) \kappa \partial_{{x}_{1}} {u} = 0 \quad \text{at } {z}_{R}, 
\end{equation*}
that is, 
\begin{equation} \label{eq212}
\partial_{{x}_{1}{x}_{2}{x}_{2}} {u} = \kappa \Delta {u} - (\kappa + \beta) \kappa \beta {u} \quad \text{at } {z}_{R}. 
\end{equation}
By \eqref{eq101} and the convexity of $\Omega$ in the ${x}_{2}$-direction, we have
\begin{equation}
\Delta {u} \leq 0 \quad \text{and} \quad \kappa({z}_{R}) \geq 0. 
\end{equation}
Hence, \eqref{eq212} yields 
\begin{equation} \label{eq213}
\partial_{{x}_{1}{x}_{2}{x}_{2}} {u} ({z}_{R}) \leq 0. 
\end{equation}

On the other hand, note that
\begin{gather} \notag
{u}_{{x}_{2}} < 0 \quad \text{in } \Omega^{ + } = \Omega \cap \{ {x}_{2} > 0 \}, 
\\ \label{eq216}
{u}_{{x}_{2}} = |\nabla {u}_{{x}_{2}}| = (\partial_{{x}_{1}})^{2} {u}_{{x}_{2}} = (\partial_{{x}_{2}})^{2} {u}_{{x}_{2}} = 0 \quad \text{at } {z}_{R}, 
\end{gather}
and the boundary $\partial\Omega^{ + }$ of the half-domain $\Omega^{ + }$ meets at a right angle at ${z}_{R}$. The identities in \eqref{eq216} follow from the even symmetry of ${u}$, the equation for ${u}_{{x}_{2}}$ and the assumption \eqref{eq211b}. Note that the ${C}^{4}$-regularity of $\partial\Omega$, together with standard
elliptic regularity, guarantees that ${u} \in {C}^{3}(\overline{\Omega})$. By applying \autoref{lemS} to the linearized equation for ${u}_{{x}_{2}}$, we obtain that either
\begin{equation*}
\partial_{\mathbf{e}} {u}_{{x}_{2}}({z}_{R}) < 0 \quad \text{or} \quad (\partial_{\mathbf{e}})^{2} {u}_{{x}_{2}}({z}_{R}) < 0, 
\end{equation*}
where $\mathbf{e} = ( - 1, 1)/\sqrt{2}$ is an inward direction relative to $\partial\Omega^{ + }$ at ${z}_{R}$. The first case violates \eqref{eq216}, and hence
\begin{equation*}
 - \partial_{{x}_{1}{x}_{2}} {u}_{{x}_{2}}({z}_{R}) < 0. 
\end{equation*}
This contradicts \eqref{eq213}, and hence the assumption \eqref{eq211b} is false. Therefore, \eqref{eq204} is proved. 
\end{proof}

Next, we show that any two domains satisfying \ref{eq01itema} and \ref{eq01itemb} can be smoothly deformed from one into another, and during the smooth deformation, \ref{eq01itema} and \ref{eq01itemb} are preserved. 

\begin{lma} \label{lma23}
Fix any constant $\beta > 0$, and let $\mathcal{F}$ denote the collection of all ${C}^{4}$ domains $\Omega$ satisfying \ref{eq01itema} and \ref{eq01itemb}. Then any two domains in $\mathcal{F}$ can be joined by a continuous family of domains in $\mathcal{F}$. 
\end{lma}

\begin{proof}
The proof is based on the observation that, after sufficiently enlarging the domains, the linear interpolation of their boundary parametrizations remains in $\mathcal{F}$. Since every disk ${B}_{\varrho} = \{{x} \in \R^{2}: |{x}| < \varrho\}$ belongs to $\mathcal{F}$, it suffices to connect an arbitrary $\Omega_{1} \in \mathcal{F}$ to the unit disk $\Omega_{0} = {B}_{1}$ by a continuous family of domains in $\mathcal{F}$. 

We claim that there exists a continuous family
$\{\widetilde{\Omega}_{{t}}\}_{{t} \in [0, 1]}$ of ${C}^{4}$ domains satisfying
\ref{eq01itema}, such that $\widetilde{\Omega}_{0} = {B}_{1}$ and $\widetilde{\Omega}_{1} = \Omega_{1}$. 
Moreover, the absolute curvatures of their boundaries are uniformly bounded; namely, 
\begin{equation} \label{eq218}
\bar{\kappa}: = \sup_{{t} \in [0, 1]} \max_{{x} \in \partial\widetilde{\Omega}_{{t}}}
\big|\kappa_{\partial\widetilde{\Omega}_{{t}}}({x})\big|
 < \infty. 
\end{equation}
Although this claim is elementary, we include the details for completeness. 
For $i = 0, 1$, let $\gamma_{i}: \mathbb{S}^{1} \to \partial\Omega_{i}$ be a counterclockwise
constant-speed parametrization, chosen so that ${s} = 0$ and ${s} = \pi$ correspond respectively to the right and left endpoints of $\partial\Omega_{i} \cap \{{x}_{2} = 0\}$. In particular, we have 
$\gamma_{0}({s}) = (\cos {s}, \sin {s})$ and
\begin{equation} \label{eq220}
\partial_{{s}} \gamma_{0}^{1}({s}) < 0\; \; \text{for} \; \; {s} \in (0, \pi). 
\end{equation}
Assumption \ref{eq01itema} can be expressed as
\begin{equation} \label{eq222a}
\begin{gathered}
\gamma_{i}( - {s}) = (\gamma_{i}^{1}({s}), - \gamma_{i}^{2}({s})) \, \text{for}\, {s} \in \mathbb{S}^{1}, \; \; \gamma_{i}^{2}(0) = \gamma_{i}^{2}(\pi) = 0, \; \; \gamma_{i}^{1}(0) > \gamma_{i}^{1}(\pi), 
\\ 
|\partial_{{s}}\gamma_{i}({s})| > 0\, \text{for} \, {s} \in \mathbb{S}^{1}, \; \gamma_{i}^{2}({s}) > 0 \geq \partial_{{s}} \gamma_{i}^{1}({s})\; \; \text{for} \; \; {s} \in (0, \pi), 
\end{gathered}
\end{equation}
which, by the parameter symmetry, consequently yields
\begin{equation} \label{eq222b}
\partial_{{s}}\gamma_{i}^{1}(0) = \partial_{{s}}\gamma_{i}^{1}(\pi) = 0, \quad \partial_{{s}}\gamma_{i}^{2}(0) > 0 > \partial_{{s}}\gamma_{i}^{2}(\pi). 
\end{equation}
We point out that the conditions $\partial_{{s}}\gamma_{i}^{1} \leq 0$ and $\gamma_{i}^{2} > 0$ on $(0, \pi)$ characterize the convexity of $\Omega_{i}$ in the ${x}_{2}$-direction. 
Define
\begin{equation} 
\gamma_{{t}}({s}) = (1 - {t})\gamma_{0}({s}) + {t}\gamma_{1}({s}), 
\quad
{s} \in \mathbb{S}^{1}, \quad {t} \in [0, 1]. 
\end{equation}
The strict inequality in \eqref{eq220} and \eqref{eq222b}, together with the preceding properties, implies that, for every ${t} \in [0, 1]$, $\gamma_{{t}}$ is a regular simple closed curve satisfying the same properties as in \eqref{eq222a}. 
Moreover, 
\begin{equation*}
 \inf_{\substack{{t} \in [0, 1], {s} \in \mathbb{S}^{1}}}
|\partial_{{s}}\gamma_{{t}}({s})| > 0, 
\quad
\sup_{{t} \in [0, 1]}
\|\gamma_{{t}}\|_{{C}^{2}(\mathbb{S}^{1})} < \infty. 
\end{equation*}
Since the signed curvature of $\gamma_{{t}}$ is
\begin{equation*} 
\kappa_{{t}}({s}) = 
\frac{\det\bigl(\partial_{{s}}\gamma_{{t}}({s}), \partial_{{s}}^{2}\gamma_{{t}}({s})\bigr)}
{|\partial_{{s}}\gamma_{{t}}({s})|^3}, 
\end{equation*}
we conclude that the absolute curvature of $\gamma_{{t}}$ is uniformly bounded by a constant $\bar{\kappa} > 0$. Let $\widetilde{\Omega}_{{t}}$ be the bounded domain enclosed by $\gamma_{{t}}$. Then $\widetilde{\Omega}_{0} = \Omega_{0}$ and $\widetilde{\Omega}_{1} = \Omega_{1}$ and every $\widetilde{\Omega}_{{t}}$ satisfies \ref{eq01itema} and \eqref{eq218}. This proves the claim.

Now choose $\varrho > \max\{1, \bar{\kappa}/\beta\}$ and define
\begin{equation*}
\Omega_{{t}} = 
\begin{cases}
\bigl(1 - 3{t} + 3{t}\varrho\bigr)\Omega_{0}, 
& {t} \in [0, 1/3], \\
\varrho\, \widetilde{\Omega}_{3{t} - 1}, 
& {t} \in [1/3, 2/3], \\
\bigl(3{t} - 2 + (3 - 3{t})\varrho\bigr)\Omega_{1}, 
& {t} \in [2/3, 1]. 
\end{cases}
\end{equation*}
The curvature condition \ref{eq01itemb} follows from $\varrho > \bar{\kappa}/\beta$ for ${t} \in [1/3, 2/3]$ and is preserved under dilation for ${t} \in [2/3, 1]$. Hence $\Omega_{{t}} \in \mathcal{F}$ for all
${t} \in [0, 1]$, and $\{\Omega_{{t}}\}_{{t} \in [0, 1]}$ is the required continuous family. 
This completes the proof. 
\end{proof}

Now we turn to complete the proof of the main result. 

\begin{proof}[Proof of \autoref{thm12main}]
Let $\Omega_{1} = \Omega$ be the target domain in \autoref{thm12main} and let us choose the unit disk $\Omega_{0} = {B}_{1}$ as the initial domain. By \autoref{lma23}, one can find a continuously varying family of domains $\{\Omega_{{t}}\}_{{t} \in [0, 1]}$ connecting $\Omega_{0}$ to $\Omega_{1}$, all satisfying \ref{eq01itema} and \ref{eq01itemb}. 
Let ${u}^{t}$ be the $\beta$-Robin torsion function in $\Omega_{{t}}$. Since $\Omega_{0}$ is the unit ball, we have
\begin{equation*}
{u}^{0}({x}) = \frac{1}{4} + \frac{1}{2\beta} - \frac{1}{4}|{x}|^{2} \quad \text{for } |{x}| \leq 1. 
\end{equation*}
It follows immediately that for ${t} = 0$ there holds
\begin{equation} \label{eq228}
\begin{gathered}
{u}^{t}({x}_{1}, {x}_{2}) = {u}^{t}({x}_{1}, - {x}_{2}) \quad \text{for } {x} \in \overline{\Omega_{{t}}}, \\
\partial_{{x}_{2}} {u}^{t} < 0 \quad \text{in } \overline{\Omega_{{t}}} \cap \{{x}_{2} > 0\}, \quad
\partial_{{x}_{2}{x}_{2}} {u}^{t} < 0 \quad \text{on } \overline{\Omega_{{t}}} \cap \{{x}_{2} = 0\}. 
\end{gathered}
\end{equation}
Recall that the map ${t} \mapsto {u}^{t}$ is continuous from $[0, 1]$ into ${C}^{2}$ function spaces. Thus, there exist small $\delta > 0$ and $\varepsilon > 0$ such that
\begin{equation*}
\partial_{{x}_{2}} {u}^{t} < 0 \quad \text{in } \overline{\Omega_{{t}}} \cap \{{x}_{2} \geq \delta\}, \quad
\partial_{{x}_{2}{x}_{2}} {u}^{t} < 0 \quad \text{on } \overline{\Omega_{{t}}} \cap \{0 \leq {x}_{2} \leq \delta\}
\end{equation*}
hold for ${t} \in [0, \varepsilon]$. Combining this with the fact that $\partial_{{x}_{2}} {u}^{{t}} = 0$ for ${x}_{2} = 0$, we deduce that \eqref{eq228} holds for ${t} \in [0, \varepsilon)$. Therefore, \eqref{eq228} holds for ${t}$ in some maximal interval $[0, \bar{t})$ for some $\bar{t} \in (0, 1]$. By continuity, we have $\partial_{{x}_{2}}{u}^{\bar{t}} \leq 0$ for ${x}_{2} > 0$
Since $\partial_{{x}_{2}} {u}^{\bar{t}} \not\equiv 0$, by applying the strong maximum principle to the (linear) equation for $\partial_{{x}_{2}} {u}^{\bar{t}}$, we obtain
\begin{equation*}
\partial_{{x}_{2}} {u}^{\bar{t}} < 0 \quad \text{in } \Omega_{\bar{t}} \cap \{{x}_{2} > 0\}. 
\end{equation*}
It then follows from \autoref{prop22} that \eqref{eq228} holds for ${t} = \bar{t}$. Again by continuity, \eqref{eq228} holds for all ${t}$ close to $\bar{t}$, so necessarily $\bar{t} = 1$, and the maximal interval is $[0, 1]$. In particular, \eqref{eq228} is valid for ${t} = 1$. This completes the proof. 
\end{proof}

\begin{rmks}
Our proof is based on the continuity method via domain deformation, whereas \cite[Theorem 8.1]{GHLZ26} derives the monotonicity directly by combining a Robin-type oblique derivative boundary inequality with the maximum principle.
\end{rmks}

We conclude this section by considering the following semilinear extension of the Robin torsion problem: 
\begin{equation} \label{eq231}
\begin{cases}
\Delta {u} + {f}({u}) = 0 & \text{in }\Omega, \\
\partial_{\nu} {u} + \beta {u} = 0 & \text{on }\partial\Omega.
\end{cases}
\end{equation}
We assume that there exists a constant $\vartheta > 0$ such that the nonlinearity ${f}$ satisfies
\begin{enumerate}[label = \rm({f}\arabic*), start = 1]
\item \label{eq234f1}
${f} \in {C}^{1, \alpha}([0, \infty))$ for some $\alpha \in (0,1)$, ${f}(\vartheta) = 0$ and ${f}({t}) < 0$ for all ${t} > \vartheta$; 
\item \label{eq234f2}
the map ${t} \longmapsto {{f}({t})}/{t}$ is strictly decreasing on $(0, \vartheta)$;
\item \label{eq234f3}
either  ${f}(0) > 0$ or ${f}(0) = 0$ and ${f}'(0) > \lambda_{1, \beta}(\Omega)$. 
\end{enumerate} 
Here $\lambda_{1, \beta}(\Omega)$ denotes the first Robin eigenvalue on $\Omega$ with Robin parameter $\beta$. In particular, \ref{eq234f1} and \ref{eq234f2} imply that ${f}({t}) > 0$ for ${t} \in (0, \vartheta)$. 

\begin{lma}[Existence and uniqueness] \label{lma24a}
Let $\Omega \subset \mathbb{R}^{n}$ be a bounded connected smooth domain, let $\beta > 0$, and suppose that ${f}$ satisfies \ref{eq234f1}--\ref{eq234f3}. Then the semilinear problem \eqref{eq231} admits a unique positive solution ${u}$. 
Moreover, it satisfies 
\begin{equation} \label{eq235}
\varepsilon_{*} \varphi_{1} < {u} < \vartheta \quad\text{in}\quad \overline{\Omega}
\end{equation}
where $\varphi_{1} > 0$ is a first Robin eigenfunction corresponding to $\lambda_{1,\beta}(\Omega)$, normalized by $\|\varphi_{1}\|_{{C}(\overline{\Omega})} = 1$, and $\varepsilon_{*} \in (0, \vartheta)$ is chosen sufficiently small so that
\begin{equation}\label{eq236}
{{f}({t})}/{t} > \lambda_{1,\beta}(\Omega) \;\; \text{for all}\;\; {t} \in (0, \varepsilon_{*}].
\end{equation}
\end{lma}

\begin{proof}
Although the result is standard, we include a simple proof based on Picone's identity. We first note that the maximum principle and the Hopf boundary lemma imply that every positive solution ${u}$ satisfies $0 < {u} < \vartheta$ in $\overline{\Omega}$. 

We next prove uniqueness. 
Let ${u}_{1}$ and ${u}_{2}$ be two positive solutions. Picone's identity gives
\begin{equation*}
|\nabla {u}_{1}|^{2} - \nabla {u}_{2}\cdot\nabla\Bigl(\frac{{u}_{1}^{2}}{{u}_{2}}\Bigr) = 
\Bigl|\nabla {u}_{1} - \frac{{u}_{1}}{{u}_{2}}\nabla {u}_{2}\Bigr|^{2}. 
\end{equation*}
Testing the equations for ${u}_{1}$ and ${u}_{2}$ with ${u}_{1}$ and ${u}_{1}^{2}/{u}_{2}$, respectively, and using the Robin boundary conditions, we obtain
\begin{equation*}
\int_{\Omega} \Bigl( {f}({u}_{1}){u}_{1} - {f}({u}_{2})\frac{{u}_{1}^{2}}{{u}_{2}} \Bigr)\, dx = 
\int_{\Omega} \Bigl|\nabla {u}_{1} - \frac{{u}_{1}}{{u}_{2}}\nabla {u}_{2}\Bigr|^{2}\, dx. 
\end{equation*}
Interchanging ${u}_{1}$ and ${u}_{2}$ and adding the resulting identities yields
\begin{equation*}
\int_{\Omega} ({u}_{1}^{2} - {u}_{2}^{2}) \Bigl( \frac{{f}({u}_{1})}{{u}_{1}} - \frac{{f}({u}_{2})}{{u}_{2}} \Bigr)\, dx = 
\int_{\Omega} \Bigl( \Big|\nabla {u}_{1} - \frac{{u}_{1}}{{u}_{2}}\nabla {u}_{2}\Big|^{2} + \Big|\nabla {u}_{2} - \frac{{u}_{2}}{{u}_{1}}\nabla {u}_{1}\Big|^{2} \Bigr)\, dx \geq 0.
\end{equation*}
By the strict monotonicity in \ref{eq234f2}, we conclude that
${u}_{1} = {u}_{2}$ in $\Omega$. This shows the uniqueness.

It remains to prove existence. Under either alternative in \ref{eq234f3} and ${f}(\vartheta) = 0$, the constant $\varepsilon_{*}$ defined in \eqref{eq236} is well-defined and satisfies $\varepsilon_{*} \in (0, \vartheta)$. We observe that $\varepsilon_{*}\varphi_{1}$ is a strict subsolution, while the constant $\vartheta$ is a strict supersolution. The sub- and supersolution method therefore yields a solution ${u}$ satisfying $\varepsilon_{*}\varphi_{1} <  {u} < \vartheta$. Together with the uniqueness proved above, this shows that the unique solution satisfies \eqref{eq235}. This completes the proof.
\end{proof}

\begin{prop} \label{prop24b}
Let $\Omega$ satisfy \ref{eq01itema} and \ref{eq01itemb}, let $\beta > 0$, and suppose that ${f}$ satisfies \ref{eq234f1}, \ref{eq234f2} and \ref{eq234f3}. If ${u}$ is a positive solution of the semilinear problem \eqref{eq231}, then ${u}$ is
symmetric with respect to $\{{x}_{2} = 0\}$ and satisfies the monotonicity properties in \eqref{eq107}. 
\end{prop}

\begin{proof}
The proof is similar to that for the torsion problem, but some additional points need to be addressed in the semilinear setting. For simplicity, after a horizontal translation, we may assume that $0 \in \Omega$. 

We construct a continuous family of admissible domains connecting a sufficiently large disk to $\Omega$, such that the semilinear problem admits a positive solution on every domain in the family. Following the proof of \autoref{lma23}, let $\{\widetilde{\Omega}_{{t}}\}_{{t} \in [0, 1/2]}$ be a continuous family of ${C}^{4}$ domains satisfying \ref{eq01itema}, with $\widetilde{\Omega}_{0} = {B}_{1}$ and $\widetilde{\Omega}_{1/2} = \Omega$. Moreover, there exist constants $\bar{\kappa}, \eta > 0$ such that $|\kappa_{\partial\widetilde{\Omega}_{{t}}}| \leq \bar{\kappa}$ and ${B}_{\eta} \subset \widetilde{\Omega}_{{t}}$ for all ${t} \in [0, 1/2]$. 
Define
\begin{equation*}
\Omega_{{t}} = 
\begin{cases}
\varrho\widetilde{\Omega}_{{t}}, 
& {t} \in [0, 1/2], \\
\bigl[(2 - 2{t})\varrho + 2{t} - 1\bigr]\Omega, 
& {t} \in [1/2, 1], 
\end{cases}
\end{equation*}
where $\varrho$ is sufficiently large so that
\begin{equation*}
\varrho \geq \max\{1, \bar{\kappa}/\beta\}
\quad\text{and}\quad
\lambda_{1}({B}_{\varrho\eta}) < \lambda_{1, \beta}(\Omega). 
\end{equation*}
Here $\lambda_{1}({B}_{\varrho\eta})$ denotes the first Dirichlet eigenvalue on ${B}_{\varrho\eta}$. Then $\{\Omega_{{t}}\}_{{t} \in [0, 1]}$ is a continuous family connecting $\Omega_{0} = {B}_{\varrho}$ to $\Omega_{1} = \Omega$, and every $\Omega_{{t}}$ satisfies \ref{eq01itema} and \ref{eq01itemb}. 
By the variational characterization and ${B}_{\varrho\eta} \subset \Omega_{{t}}$ for ${t} \in [0, 1/2]$, and by the
monotonicity of ${s} \mapsto \lambda_{1, \beta}({s}\Omega)$ for ${t} \in [1/2, 1]$, we obtain 
\begin{equation*}
\lambda_{1, \beta}(\Omega_{{t}}) \leq \lambda_{1, \beta}(\Omega) \quad\text{for all }{t} \in [0, 1]. 
\end{equation*}
Thus, by the existence criterion above, the semilinear problem admits a positive solution ${u}^{{t}}$ in $\Omega_{{t}}$ for every ${t} \in [0, 1]$. 

By \autoref{lma24a}, ${u}^{{t}}$ is unique and satisfies $\varepsilon_{*}\varphi_{{t}} < {u}^{{t}} < \vartheta$ with $\varphi_{{t}}$ is ${L}^{\infty}$ normalized positive eigenfunction to $\lambda_{1, \beta}(\Omega_{{t}})$. Since each $\Omega_{{t}}$ is symmetric with respect to $\{{x}_{2} = 0\}$, uniqueness also implies that ${u}^{{t}}$ has the same
symmetry. After pulling the problems back to fixed reference domains, standard elliptic estimates yield compactness, while uniqueness identifies the limit. Hence ${u}^{{t}}$ depends continuously on ${t}$. 

On the disk $\Omega_{0} = {B}_{\varrho}$, uniqueness implies that ${u}^{0}$ is radial. Its radial equation yields
\begin{equation*}
\bigl({r}^{n - 1}({u}^{0})'({r})\bigr)'
 = - {r}^{n - 1}{f}({u}^{0}({r})) < 0, 
\end{equation*}
hence ${u}_{0}'({r}) < 0$ for $0 < {r} \leq \varrho$. 

The proofs of \autoref{thm12main} and \autoref{prop22} use precisely the positivity and superharmonicity of the solution, the initial radial monotonicity on the disk, and continuous dependence under domain deformation. The same argument therefore yields the monotonicity properties in \eqref{eq107}.
\end{proof}

\begin{rmks}
The symmetry in \autoref{prop24b} follows from uniqueness, ensured by the strict decrease of ${t} \mapsto {f}({t})/{t}$, rather than from a moving-plane argument. This does not contradict the failure of Gidas--Ni--Nirenberg-type symmetry for Robin problems; see \cite{CMP24}. Related symmetry-breaking results for mixed boundary conditions can be found in \cite{CPP24}.
\end{rmks}


\section{Some counterexamples} \label{Sec3Example}

In this section, we demonstrate that the monotonicity property \eqref{eq107} in \autoref{thm12main} fails in general when only condition \ref{eq01itema} is imposed. 

A typical example in which \ref{eq01itemb} fails is when the boundary curvature is sufficiently negative. 
Accordingly, we first consider a piecewise smooth domain $\Omega$ with a nonsmooth boundary point ${p}$ whose interior angle exceeds $\pi$; see, e.g., \autoref{fig02}. 
Without additional assumptions, in a neighborhood of ${p}$ the Robin torsion function is, in general, not ${C}^{1}$ and not monotonically decreasing in either coordinate variable, as indicated by a local expansion. Motivated by this observation, we rigorously verify the following monotonicity and non-monotonicity properties on a half-domain for the specific nonconvex polygon $\mathcal{P}$ below. 

\begin{prop} \label{prop31}
Let ${a}, {b} \in (1, \infty)$ be two constants, and define
\begin{equation} \label{eq302}
\begin{aligned}
\mathcal{R}_{{a}}^{h} & = \{ {x} \in \R^{2}: \, |{x}_{1}| < {a}, \, |{x}_{2}| < 1 \}, \\
\mathcal{R}_{{b}}^{v} & = \{ {x} \in \R^{2}: \, |{x}_{1}| < 1, \, |{x}_{2}| < {b} \}, \\
\mathcal{P} & = \mathcal{P}_{{a}, {b}} = \mathcal{R}_{{a}}^{h} \cup \mathcal{R}_{{b}}^{v}. 
\end{aligned}
\end{equation}
See \autoref{fig02}. 
Let $\beta \in (0, \infty)$ and let ${u}$ be the $\beta$-torsion function in $\mathcal{P}$, i.e., 
\begin{equation*}
\Delta {u} = - 1 \quad \text{in } \mathcal{P}, \quad \partial_{\nu} {u} + \beta {u} = 0 \quad \text{on } \partial \mathcal{P}. 
\end{equation*}
Then ${u}$ is symmetric with respect to both coordinate axes. More precisely, 
\begin{enumerate}[label = \rm(\arabic*)]
\item \label{eq3itemA}
If ${a} = {b}$, then 
\begin{equation} \label{eq303}
\partial_{{x}_{2}} {u} < 0 \text{ in } \mathcal{P} \cap \{ {x}_{2} > 0 \}. 
\end{equation}
\item \label{eq3itemB}
If ${a} < {b}$, then \eqref{eq303} \textbf{does not} hold; in fact, $\partial_{{x}_{2}} {u} > 0$ at every point of $\mathcal{P} \cap \{ {x}_{1} \leq 1\}$ sufficiently close to the nonconvex vertex $(1, 1)$.
\item 
One has $\partial_{{x}_{1}} {u} < 0$ in $\mathcal{P} \cap \{ {x}_{1} > 0 \}$ whenever ${a} \leq {b}$. 
\end{enumerate}
\end{prop}

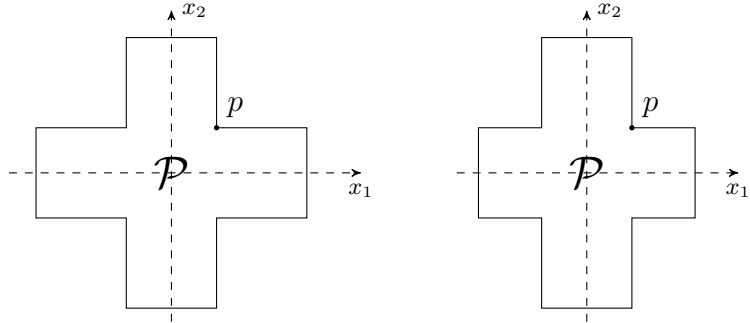
\begin{figure}[htp]\centering
\begin{tikzpicture}[scale = 1]
\pgfmathsetmacro\la{0.6}; \pgfmathsetmacro\lb{\la}; \pgfmathsetmacro\ld{\la*3}; \pgfmathsetmacro\lc{\ld}; 
\draw[] (\la, \lb) -- (\la, \ld) -- ( - \la, \ld) -- ( - \la, \lb) -- ( - \lc, \lb) -- ( - \lc, - \lb) -- ( - \la, - \lb) -- ( - \la, - \ld) -- (\la, - \ld) -- (\la, - \lb) -- (\lc, - \lb) -- (\lc, \lb) -- cycle; 
\node at (0, 0) {\Large $\mathcal{P}$}; 
\fill (\la, \lb) circle (1pt) node [above right] {${p}$}; 
\draw [dashed, -> ] ( - \lc*1.2, 0) -- (\lc*1.4, 0) node [below] {\scriptsize${x}_{1}$}; 
\draw [dashed, -> ] (0, - \ld*1.1) -- (0, \ld*1.2) node [right] {\scriptsize${x}_{2}$}; 
\end{tikzpicture}
\quad\quad
\begin{tikzpicture}[scale = 1]
\pgfmathsetmacro\la{0.6}; \pgfmathsetmacro\lb{\la}; \pgfmathsetmacro\ld{\la*3}; \pgfmathsetmacro\lc{\ld*0.8}; 
\draw[] (\la, \lb) -- (\la, \ld) -- ( - \la, \ld) -- ( - \la, \lb) -- ( - \lc, \lb) -- ( - \lc, - \lb) -- ( - \la, - \lb) -- ( - \la, - \ld) -- (\la, - \ld) -- (\la, - \lb) -- (\lc, - \lb) -- (\lc, \lb) -- cycle; 
\node at (0, 0) {\Large $\mathcal{P}$}; 
\fill (\la, \lb) circle (1pt) node [above right] {${p}$}; 
\draw [dashed, -> ] ( - \lc*1.2, 0) -- (\lc*1.4, 0) node [below] {\scriptsize${x}_{1}$}; 
\draw [dashed, -> ] (0, - \ld*1.1) -- (0, \ld*1.2) node [right] {\scriptsize${x}_{2}$}; 
\end{tikzpicture} \vspace*{ - 1ex}
\caption{Nonconvex polygon $\mathcal{P}$} \label{fig02}
\end{figure}

\begin{proof}
Let $\Gamma_{h}$ and $\Gamma_{v}$ be the horizontal and vertical boundaries, respectively. That is, 
\begin{equation}
\Gamma_{h} = \{ {x} \in \partial\mathcal{P}: \, |{x}_{2}| = 1 \text{ or } {b} \}, \quad
\Gamma_{v} = \{ {x} \in \partial\mathcal{P}: \, |{x}_{1}| = 1 \text{ or } {a} \}. 
\end{equation}
Let
\begin{equation}
\mathcal{P}^{ ++ } = \mathcal{P} \cap \{ {x}_{1} > 0, \, {x}_{2} > 0 \}. 
\end{equation}
The proof is divided into several parts. 

\textbf{Part~1}. 
Local behavior of ${u}$ near the vertices.
Standard elliptic regularity implies that $u$ is smooth up to the boundary away from the four reentrant vertices $(\pm1, \pm1)$. 
By uniqueness and the symmetry of the problem, $u$ is symmetric with respect to both coordinate axes. Thus, it suffices to consider the reentrant vertex ${p} = (1, 1)$. 
Set
\begin{equation*}
{y} = ({y}_{1}, {y}_{2}) = ({x}_{1} - 1, {x}_{2} - 1), \quad {r} = |{y}|, \quad 
\theta = \arg({y}) - 2\pi \in ( - 3\pi/2, 0). 
\end{equation*}
The tangent sector of $\mathcal{P}$ at ${p}$ is \({W} = \{({r}, \theta): r > 0, \; - 3\pi/2 < \theta < 0\}, \)
whose opening is $\omega = 3\pi/2$. On the two boundary rays, the Robin conditions are
\begin{equation*}
{B}_{ + }{u}: = \frac{1}{{r}}\partial_{\theta} {u} + \beta {u} = 0 \; \; \text{at}\; \; \theta = 0, 
\quad
{B}_{ - }{u}: = - \frac{1}{{r}}\partial_{\theta} {u} + \beta {u} = 0 \; \; \text{at}\; \; \theta = - \tfrac{3\pi}{2}. 
\end{equation*}
Equivalently, after multiplication by ${r}$, 
\begin{equation*}
\partial_{\theta} {u} + \beta {r} {u} = 0\; \; \text{at}\; \; \theta = 0, 
\quad
 - \partial_{\theta} {u} + \beta {r} {u} = 0 \; \; \text{at}\; \; \theta = - 3\pi/2. 
\end{equation*}

The general corner decomposition for elliptic boundary value problems on polygonal domains follows from \cite[Chapter 4, Section 4.4, Theorem 4.4.4.13, pp.~245--247]{Gri85}; see also \cite[Chapter 5, Section 5.1, Proposition 5.1.1.3 and Theorem 5.1.1.4, pp.~254--256]{Gri85} for higher-order expansions. For the treatment of lower-order terms in the interior and boundary operators and the resulting shadow terms, we refer to
\cite[Chapter 6, Section 6.4, Theorem 6.4.2 and Corollary 6.4.1, pp.~237--240]{KMR97}. 

The principal part of the Robin boundary operator at the corner is the Neumann operator. Hence the nonnegative corner exponents and the corresponding homogeneous Neumann modes are
\begin{equation*}
\lambda_{j} = \frac{{j}\pi}{\omega} = \frac{2{j}}{3}, 
\quad
\mathcal{S}_{j}({r}, \theta) = {r}^{\lambda_{j}}\cos(\lambda_{j}\theta), 
\quad {j} = 0, 1, 2, \ldots. 
\end{equation*}
The term $\beta {r} {u} $ in the rescaled boundary condition is one order higher in ${r}$. It therefore does not change the principal corner exponents, but produces higher-order shadow terms. In particular, the first shadow associated with $\lambda_{1} = 2/3$ has exponent $\lambda_{1} + 1 = 5/3$. 
Since $5/3$ is not an eigenvalue of the principal Neumann pencil, no logarithmic term occurs at this order. 

We now compute this first shadow explicitly. To cancel the boundary residual $\beta\mathcal{S}_{1}$, we seek the first Robin correction in the form 
\begin{equation*}
\tilde{\mathcal{S}}_{1}({r}, \theta)
 = \beta {r}^{5/3}{f}(\theta). 
\end{equation*}
The condition $\Delta\tilde{\mathcal{S}}_{1} = 0$ and the cancellation of the boundary residual of order ${r}^{2/3}$, equivalently, ${B}_{\pm} (\mathcal{S}_{1} + \tilde{\mathcal{S}}_{1}) = O({r}^{5/3})$, give
\begin{equation*}
f'' + \frac{25}{9}f = 0, 
\quad
f'(0) = f'( - \frac{3\pi}{2}) = - 1. 
\end{equation*}
Solving this elementary boundary value problem gives 
\begin{equation*}
\tilde{\mathcal{S}}_{1}({r}, \theta)
 = 
\frac{3\sqrt2}{5}\beta {r}^{5/3}
\cos\bigl(\frac{5\theta}{3} + \frac{3\pi}{4}\bigr). 
\end{equation*}
The regular branch associated with the zero corner exponent begins with
\begin{equation*}
{Q}_{0}({y})
 = 
1 - \beta({y}_{1} + {y}_{2}) + \beta^{2}{y}_{1}{y}_{2}
 = 
(1 - \beta {y}_{1})(1 - \beta {y}_{2}). 
\end{equation*}
This polynomial is harmonic and satisfies the Robin conditions exactly on the two boundary rays. Moreover, since $\Delta {u} = - 1$, we introduce the quadratic particular solution
\begin{equation*}
{Q}_{p}({y}) = - \frac{1}{4}{r}^{2}, 
\quad
\Delta {Q}_{p} = - 1. 
\end{equation*}
Its Robin boundary residual is of order ${r}^{2}$ and is corrected by terms of order ${r}^3$, which are included in the remainder. 

Consequently, there exist constants ${c}_{0} = {u}({p}) > 0$ and ${c}_{1}, {c}_{2}, {c}_{3} \in \R$ such that, for ${k} = 0, 1, 2$, 
\begin{equation*}
|D^{k}({u} - \Phi)({y})|
 = O\big({r}^{7/3 - {k}}\big)
\quad\text{as }{r} \to 0, 
\end{equation*}
where
\begin{equation} \label{eq305}
\begin{aligned}
\Phi({y}) = {}&
{c}_{0}[ 1 - \beta({y}_{1} + {y}_{2}) + \beta^{2}{y}_{1}{y}_{2}]
 - \frac{1}{4}{r}^{2}
\\
& + {c}_{1}\Big[
{r}^{2/3}\cos\big(\frac{2\theta}{3}\big)
 + \frac{3\sqrt2}{5}\beta {r}^{5/3}
\cos\big(\frac{5\theta}{3} + \frac{3\pi}{4}\big) \Big]
\\
& + {c}_{2}{r}^{4/3}\cos\big(\frac{4\theta}{3}\big)
 + {c}_{3}{r}^{2}\cos(2\theta). 
\end{aligned}
\end{equation}
The first Robin shadow of the ${r}^{4/3}$ mode is of order ${r}^{7/3}$ and is therefore included in the remainder. 

If ${a} = {b}$, then ${u}$ is symmetric under the reflection $({y}_{1}, {y}_{2})\mapsto({y}_{2}, {y}_{1})$, whereas the ${r}^{2/3}$ corner mode is antisymmetric. Hence ${c}_{1} = 0$. 

\textbf{Part~2}. 
We claim that ${u}$ is not ${C}^{1}$ at the nonconvex vertices (i.e., ${c}_{1} \neq 0$) when ${a} \neq {b}$. 
Suppose, for contradiction, that ${a} < {b}$ and ${u}$ is ${C}^{1}$ up to the boundary, that is, 
\begin{equation} \label{eq308}
{c}_{1} = 0. 
\end{equation}
Let ${v} = \partial_{{x}_{2}}{u} + \beta {u}$. By \eqref{eq308}, ${v}$ is continuous up to the boundary and satisfies
\begin{equation}
\begin{cases}
 - \Delta {v} = \beta & \text{in } \mathcal{P}^{ ++ }, \\
{v} = 0 & \text{on } \partial\mathcal{P}^{ ++ } \cap \{ {x}_{2} = 1, \, {b} \}, \\
{v} > 0 & \text{on } \partial\mathcal{P}^{ ++ } \cap \{ {x}_{2} = 0 \}, \\
\partial_{\nu}{v} + \beta {v} = 0 & \text{on } \partial\mathcal{P}^{ ++ } \cap \{ {x}_{1} = 1, \, {a} \}, \\
\partial_{\nu}{v} = 0 & \text{on } \partial\mathcal{P}^{ ++ } \cap \{ {x}_{1} = 0 \}. 
\end{cases}
\end{equation}
By the maximum principle and Hopf lemma, 
\begin{equation} \label{eq310}
{v} = \partial_{{x}_{2}}{u} + \beta {u} > 0 \quad \text{in } \overline{\mathcal{P}^{ ++ }} \setminus \Gamma_{h}. 
\end{equation}

Let $\mathcal{D} = \mathcal{P}^{ ++ } \cap \{ {x}_{2} < {x}_{1} \}$ and
\begin{equation} \label{eq311a}
{w}({x}_{1}, {x}_{2}) = {u}({x}_{1}, {x}_{2}) - {u}({x}_{2}, {x}_{1}), \quad {x} \in \mathcal{D}. 
\end{equation}
By \eqref{eq310}, we get
\begin{equation} \label{eq311}
\begin{cases}
 - \Delta {w} = 0 & \text{in } \mathcal{D}, \\
{w} = 0 & \text{on } \partial\mathcal{D} \cap \{ {x}_{1} = {x}_{2} \}, \\
\partial_{\nu}{w} = 0 & \text{on } \partial\mathcal{D} \cap \{ {x}_{2} = 0 \}, \\
\partial_{\nu}{w} + \beta {w} < 0 & \text{on } \partial\mathcal{D} \cap \{ {x}_{1} = {a} \}, \\
\partial_{\nu}{w} + \beta {w} = 0 & \text{on } \partial\mathcal{D} \cap \{ {x}_{2} = 1 \}. 
\end{cases}
\end{equation}
The maximum principle implies that ${w} < 0$ in $\mathcal{D}$, and the Hopf lemma implies
\begin{equation} \label{eq312}
{u}({x}_{1}, {x}_{2}) < {u}({x}_{2}, {x}_{1}), \quad {x} \in \overline{\mathcal{D}} \setminus \{ {x}_{1} = {x}_{2} \}. 
\end{equation}
From \eqref{eq308} and symmetry, we get that the difference function ${w}$ is of class ${C}^{2}(\overline{\mathcal{D}})$, and
\begin{equation} \label{eq314}
{w} = 2{c}_{3}({y}_{1}^{2} - {y}_{2}^{2}) + O ( |{y}|^{7/3} ). 
\end{equation}
Since the interior angle of $\mathcal{D}$ at ${p}$ is obtuse, applying Serrin's corner lemma (\autoref{lemS}) to ${w}$ in $\mathcal{D}$ yields
\begin{equation*}
\nabla {w} \cdot (0, - 1) < 0 \quad \text{at } {y} = 0. 
\end{equation*}
This yields a contradiction to \eqref{eq314}. Hence, Part~2 follows. 

\textbf{Part~3}. 
${u}$ satisfies
\begin{equation} \label{eq316}
\partial_{{x}_{2}} ( \partial_{{x}_{2}}{u} + \beta {u} ) < 0 \quad \text{on } \overline{\mathcal{P}} \cap \{ {x}_{2} = {b} \} \text{ when } {b} = {a}. 
\end{equation}
Indeed, since ${a} = {b}$, ${u}$ is symmetric with respect to both lines $\{ {x}_{1} = \pm {x}_{2} \}$, so the coefficient ${c}_{1}$ in \eqref{eq305} vanishes. Hence, similar to the derivation of \eqref{eq310}, 
\begin{equation} \label{eq317}
\begin{aligned}
{v} & = \partial_{{x}_{2}}{u} + \beta {u} > 0 \quad \text{in } ( \overline{\mathcal{P}} \setminus \Gamma_{h} ) \cap \{ {x}_{2} \geq 0 \}, 
\\
{v} & = \partial_{{x}_{2}}{u} + \beta {u} = 0 \quad \text{for } {x}_{2} = {b}, 
\end{aligned}
\end{equation}
and the superharmonic function ${v}$ attains its minimum zero along ${x}_{2} = {b}$. Applying the Hopf lemma to ${v}$, we get
\begin{equation} \label{eq319}
\partial_{{x}_{2}} ( \partial_{{x}_{2}}{u} + \beta {u} ) < 0 \quad \text{for } {x}_{2} = {b}, \, {x}_{1} \in ( - 1, 1). 
\end{equation}
Now consider \eqref{eq316} at the corner $(1, {b})$. From the Robin boundary condition for ${u}$, we have $\partial_{{x}_{2}} {u} + \beta {u} = 0$ on ${x}_{2} = {b}$. This gives
\begin{equation*}
\partial_{{x}_{1}{x}_{1}} ( \partial_{{x}_{2}} {u} + \beta {u} ) = 0 \quad \text{for } {x}_{2} = {b}, \, {x}_{1} \in [ - 1, 1]. 
\end{equation*}
Combining this with the equation for ${u}$, we have
\begin{equation*}
\partial_{{x}_{2}{x}_{2}} ( \partial_{{x}_{2}} {u} + \beta {u} ) = \Delta ( \partial_{{x}_{2}} {u} + \beta {u} ) = - \beta < 0 \quad \text{for } {x}_{2} = {b}, \, {x}_{1} \in [ - 1, 1]. 
\end{equation*}
Combining this with \eqref{eq317}, we see that \eqref{eq316} also holds at the corner $(1, {b})$. 

\textbf{Part~4}. 
We claim that the coefficient ${c}_{1}$ in \eqref{eq305} is negative when $1 < {a} < {b}$. 

Indeed, fix ${b} \in (1, \infty)$ and vary ${a} \in (1, {b})$. By Part~3 and continuity, there exists a small constant $\epsilon > 0$ such that if ${a} \in ({b} - \epsilon, {b})$, then the corresponding solution ${u}$ satisfies
\begin{equation*}
\partial_{{x}_{2}} ( \partial_{{x}_{2}}{u} + \beta {u} ) < 0 \quad \text{on } \overline{\mathcal{P}} \cap \{ {a} \leq {x}_{2} \leq {b} \}. 
\end{equation*}
Combining this with $\partial_{{x}_{2}}{u} + \beta {u} = 0$ at ${x}_{2} = {b}$, we have
\begin{equation}
\partial_{{x}_{2}}{u} + \beta {u} > 0 \quad \text{on } \overline{\mathcal{P}} \cap \{ {a} \leq {x}_{2} < {b} \}. 
\end{equation}
It follows that the difference function ${w}$ in \eqref{eq311a} is continuous and satisfies \eqref{eq311}. Hence, we obtain
\begin{equation*}
{w}({x}_{1}, {x}_{2}) = {u}({x}_{1}, {x}_{2}) - {u}({x}_{2}, {x}_{1}) < 0 \quad \text{for } {x} \in \overline{\mathcal{P}} \text{ with } {x}_{1} > {x}_{2} \geq 0. 
\end{equation*}
From Parts 1 and 2, 
\begin{gather*}
{w}({y}) = 2{c}_{1}{r}^{2/3}\cos( \tfrac{2}{3}\theta ) + O({r}^{5/3}), \\
{c}_{1} \neq 0 \text{ whenever } {a} \neq {b}. 
\end{gather*}
Therefore, ${c}_{1} < 0$ when ${a} \in ({b} - \epsilon, {b})$. As the coefficient ${c}_{1}$ depends continuously on ${a}$, we deduce that ${c}_{1} < 0$ holds whenever $1 < a < b$. 

\textbf{Part~5}. 
The monotonicity property. 
From Part~1, the derivatives of ${u}$ near the vertex ${p} = (1, 1)$ have the following properties: 
\begin{equation} \label{eq321}
\begin{aligned}
{u}_{{x}_{1}} = &\, \tfrac{2}{3} {c}_{1} {r}^{ - 1/3} \cos( \tfrac{1}{3} \theta ) - {c}_{0} \beta + O({r}^{1/3}), \\
{u}_{{x}_{2}} = &\, \tfrac{2}{3} {c}_{1} {r}^{ - 1/3} \sin( \tfrac{1}{3} \theta ) - {c}_{0} \beta + O({r}^{1/3}). 
\end{aligned}
\end{equation}

We first consider the case when ${a} < {b}$. By Part~4, ${c}_{1} < 0$. Combining this with \eqref{eq321}, we find a neighborhood $\mathcal{O}_{p}$ of ${p} = (1, 1)$ such that
\begin{equation}
{u}_{{x}_{1}} < 0 \quad \text{in } \mathcal{P} \cap \mathcal{O}_{p} \quad \text{and} \quad
{u}_{{x}_{2}} > 0 \quad \text{in } (\overline{\mathcal{P}} \cap \mathcal{O}_{p} \cap \{{x}_{1} \leq 1\}) \setminus \{{p}\}. 
\end{equation}
This concludes the proof of \ref{eq3itemB}. 
Recall that ${u}_{{x}_{1}}$ satisfies
\begin{equation} \label{eq323}
\begin{cases}
 - \Delta {u}_{{x}_{1}} = 0 & \text{in } \mathcal{P}^{ ++ }, \\
{u}_{{x}_{1}} < 0 & \text{on } \partial\mathcal{P}^{ ++ } \cap \{ {x}_{1} = 1, \, {a} \}, \\
{u}_{{x}_{1}} = 0 & \text{on } \partial\mathcal{P}^{ ++ } \cap \{ {x}_{1} = 0 \}, \\
\partial_{\nu}{u}_{{x}_{1}} + \beta {u}_{{x}_{1}} = 0 & \text{on } \partial\mathcal{P}^{ ++ } \cap \{ {x}_{2} = 1, \, {b} \}, \\
\partial_{\nu}{u}_{{x}_{1}} = 0 & \text{on } \partial\mathcal{P}^{ ++ } \cap \{ {x}_{2} = 0 \}. 
\end{cases}
\end{equation}
By applying the maximum principle to ${u}_{{x}_{1}}$ in $\mathcal{P}^{ ++ } \setminus \mathcal{O}_{p}$, we conclude that ${u}_{{x}_{1}}$ is negative in $\mathcal{P}^{ ++ } \setminus \mathcal{O}_{p}$, and hence 
\begin{equation} \label{eq324a}
\partial_{{x}_{1}}{u} < 0 \quad \text{in } \mathcal{P} \cap \{ {x}_{1} > 0 \}. 
\end{equation}

It remains for us to consider the case when ${a} = {b}$. Then ${c}_{1} = 0$, and ${u}$ is ${C}^{1}$ up to the boundary. 
We may therefore apply the maximum principle directly to \eqref{eq323} in $\mathcal{P}^{ ++ }$ to obtain \eqref{eq324a}. By symmetry, $\partial_{{x}_{2}}{u} < 0$ in $\mathcal{P} \cap \{ {x}_{2} > 0 \}$. This proves \ref{eq3itemA}. 

The preceding argument also establishes monotonicity in the ${x}_{1}$-direction for all $1 < {a} \leq {b}$; namely, \eqref{eq324a} holds whenever $1 < {a} \leq {b}$. This completes the proof. 
\end{proof}

We now present the smooth counterexample, obtained via a small perturbation of the polygonal configuration analyzed in \autoref{prop31}. 

\begin{proof}[Proof of \autoref{thm13Example}]
Let ${v}$ be the $\beta$-torsion function on the nonconvex polygon $\mathcal{P} = \mathcal{P}_{{a}, {b}}$ with any fixed ${b} > {a} > 1$. By \autoref{prop31}, $\partial_{{x}_{2}}{v} > 0$ near the corner $(1, 1)$. Therefore, there exist a small constant $\bar{\delta} > 0$, and two distinct points $\bar{p} = (1 - \bar{\delta}, \, 1 + \bar{\delta})$ and $\bar{q} = (1 - \bar{\delta}, 1 - \bar{\delta})$ in $\mathcal{P}$ such that
\begin{equation} \label{eq328a}
{v}(\bar{p}) > {v}(\bar{q}). 
\end{equation}
Let $3\varepsilon = {v}(\bar{p}) - {v}(\bar{q}) > 0$, and let ${K}$ be any fixed compact subset of $\mathcal{P}$ containing both $\bar{p}$ and $\bar{q}$. One can find a smooth domain $\Omega$ (satisfying \ref{eq01itema}) that is close to $\mathcal{P}$ such that the $\beta$-torsion function ${u}$ on $\Omega$ satisfies $\| {u} - {v} \|_{C({K})} < \varepsilon$. Combining this with \eqref{eq328a}, we get
\begin{equation} 
{u}(\bar{p}) - {u}(\bar{q}) > 0. 
\end{equation}
Consequently, there exists a point in the upper portion of the domain at which $\partial_{{x}_{2}}{u} > 0$. This completes the proof. 
\end{proof}


\section*{Acknowledgments}
Research of Qinfeng Li was supported by National Natural Science Fund of China for Excellent Young Scholars (No. 12522109), National Key R\&D Program of China (2022YFA1006900) and the National Science Fund of China General Program (No. 12471105). Research of Juncheng Wei was supported by GRF fund of RGC of Hong Kong entitled ``New frontiers in singularity formations of nonlinear partial differential equations''. Research of Ruofei Yao was supported by Guangdong Basic and Applied Basic Research Foundation (Grant No. 2025A1515011856) and National Natural Science Foundation of China (No. 12671237).

\end{document}